\documentclass{amsart}
\usepackage{amsmath}
\usepackage{amsfonts}
\usepackage{amsthm}
\usepackage{amssymb}
\usepackage{esint}
\usepackage{times}

\newtheorem{theorem}{Theorem}
\newtheorem*{theorem*}{Theorem}
\numberwithin{equation}{section}
\numberwithin{theorem}{section}

\newtheorem*{acknowledgement*}{Acknowledgement}
\newtheorem*{definition*}{Definition}

\newtheorem{lemma}[theorem]{Lemma}

\newtheorem{proposition}[theorem]{Proposition}

\newtheorem*{question*}{Question}

\newcommand{\RR}[0]{\mathbb{R}}

\newcommand{\pd}[2]{\frac{\partial #1}{\partial#2}}

\newcommand{\pdt}[0]{\frac{\partial}{\partial t}}

\newcommand{\gt}[0]{\tilde{g}}
  \newcommand{\ft}[0]{\tilde{f}}

\newcommand{\delb}[0]{\overline{\nabla}}
\newcommand{\delt}[0]{\widetilde{\nabla}}

\newcommand{\Rc}[0]{\operatorname{Rc}}
\newcommand{\Rm}[0]{\operatorname{Rm}}

\newcommand{\dfn}[0]{\doteqdot}

\newcommand{\Lc}[0]{\mathcal{L}}

\newcommand{\Ec}[0]{\mathcal{E}}

\newcommand{\Wc}[0]{\mathcal{W}}

\newcommand{\dv}[0]{\operatorname{div}}
\newcommand{\Diff}[0]{\operatorname{Diff}}

\title[The maximal rate of convergence of the Ricci flow]{On the maximal rate of convergence \\ under the Ricci flow}

 \author{Brett Kotschwar}
 \email{kotschwar@asu.edu}
 \address{School of Mathematical and Statistical Sciences,
 	Arizona State University, Tempe, AZ 85287, USA}

  \thanks{The author was partially supported by Simons Foundation grant \#359335.}

\begin{document}
\begin{abstract}
We estimate from above the rate at which a solution to the normalized Ricci flow on a closed manifold may converge to a limit soliton. Our main result
implies that any solution which converges modulo diffeomorphisms to a soliton faster than any fixed exponential rate must itself be self-similar.
\end{abstract}
\maketitle

\section{Introduction} Let $M$ be a compact manifold. We will consider solutions to the
\emph{normalized Ricci flow}
\begin{equation}\label{eq:nrf}
 \pd{g}{t} = -2\left(\Rc(g) + \frac{\sigma}{2}g\right)
\end{equation}
on $M$ for a fixed $\sigma\in \{-1, 0, 1\}$. This equation is homothetically equivalent to the usual Ricci flow: when $\sigma \neq 0$,
the transformations
\[
\tilde{g}(t) = (1 + \sigma t)g\left(\sigma^{-1}\log(1+\sigma t)\right), \quad g(t) = e^{-\sigma t}\gt\left(\sigma^{-1}(e^{\sigma t} -1 )\right),
\]
convert a solution $g(t)$ of \eqref{eq:nrf} into a solution $\tilde{g}(t)$ of the Ricci flow and vice-versa.

On a compact manifold $M$, the fixed points of the normalized Ricci flow modulo the action of $\operatorname{Diff}(M)$ are \emph{gradient Ricci solitons}, satisfying 
\begin{equation}\label{eq:grs}
 \Rc(g) + \nabla\nabla f + \frac{\sigma}{2}g = 0
\end{equation}
for some $f\in C^{\infty}(M)$. In fact, compact steady and expanding solitons are necessarily Einstein \cite{Ivey3DSolitons},  so the function $f$ 
may be taken to be constant
when $\sigma \geq 0$.  It will be convenient nevertheless to use \eqref{eq:grs} to describe all three cases in a unified way.   

In this paper, we will consider the Ricci flow in more-or-less ideal circumstances, in which a solution to the normalized equation \eqref{eq:nrf} exists for all $t\in [0, \infty)$, and converges, modulo diffeomorphisms, to a limit satisfying \eqref{eq:grs}.  In such circumstances, one can ask at what rate the convergence takes place. 
Our main result shows that this rate is at most exponential unless the solution is self-similar.

\begin{theorem}\label{thm:decay}
Let $(M, \bar{g})$ be a  closed Ricci soliton satisfying \eqref{eq:grs} and   $g(t)$ a smooth solution to \eqref{eq:nrf} 
on $M\times [0, \infty)$ for some fixed $\sigma\in \{-1, 0, 1\}$.
Suppose that there is a sequence of times $t_i \to \infty$ and $\phi_i \in \operatorname{Diff}(M)$ such that 
$\phi_i^*g(t_i)\to \bar{g}$ in $C^k_{\bar{g}}$ for all $k \geq 0$. Then, either
 \begin{equation}\label{eq:decay}
 \|\phi^*_ig(t_i) - \bar{g}\|_{C^2_{\bar{g}}} \geq C e^{-mt_i} 
\end{equation}
for some $m > 0$ and $C> 0$ or there exists a smooth family $\Phi_t\in \Diff(M)$ such that $\Phi_t^*g(t) = \bar{g}$ for all $t\geq 0$.
\end{theorem}

Of course, associated to each soliton $(M, \bar{g}, \bar{f}, \sigma)$ satisfying \eqref{eq:grs}, there is a self-similar solution $g(t) = \psi_t^*\bar{g}$ to \eqref{eq:nrf} defined for $t \in [0, \infty)$
where $\psi_t$ is the one-parameter family of diffeomorphisms generated by $\bar{\nabla}\bar{f}$. Taking $\phi_t = \psi_t^{-1}\circ \theta_t$
for different families of diffeomorphisms $\theta_t$ converging to an isometry of $\bar{g}$, we may prescribe the rate at which $\phi_t^*g(t) = \theta_t^*\bar{g}$ converges to $\bar{g}$ (or, indeed, arrange for it to coincide with $\bar{g}$ for all $t$). The theorem asserts that sequential convergence at a super-exponential rate can only occur
in this way, that is, for a solution moving exclusively within a fixed  $\operatorname{Diff}(M)$-orbit of $\bar{g}$. 

At the same time, there may be solutions to \eqref{eq:nrf} which converge to a general soliton $\bar{g}$
at arbitrarily high (but fixed) exponential rates, so the dichotomy asserted by the theorem is optimal in a sense.

Under the correspondence between the normalized and unnormalized Ricci flows described above, 
a maximal solution $\gt(t)$ to the unnormalized Ricci flow  on $M\times [0, 1)$ satisfying the Type-I curvature and diameter bounds
\[
(1-t)|\Rm(\gt(t))| \leq C, \quad \operatorname{diam}(M, \gt(t)) \leq C\sqrt{1-t},
\]
corresponds to an immortal solution $g(t)$ to \eqref{eq:nrf} with $\sigma = -1$ satisfying the uniform bounds $|\Rm(g(t))| \leq C$ and $\operatorname{diam}(M, g(t)) \leq C$ for $t\geq 0$. For such a solution $g(t)$, \v{S}e\v{s}um \cite{Sesum} proved that for any sequence $t_i \to \infty$, there is a sequence $\phi_i$ of diffeomorphisms such that $\phi_i^*g(t_i)$ subconverges to a shrinking soliton $\bar{g}$. Provided at least one of the limit solitons $\bar{g}$ is integrable, she proved 
that it is unique,  up to diffeomorphisms, and that there is a smooth family of diffeomorphisms $\psi_t$ for which $\psi_t^*g(t)$ converges smoothly to $\bar{g}$ as $t\to \infty$ at a rate that is \emph{at least} exponential.  Using an approach of Sun-Wang \cite{SunWang}, the integrability condition was later removed by Ache \cite{Ache}, however, the convergence in the general case is only guaranteed to occur at a polynomial rate. (It is expected that there are solutions to \eqref{eq:nrf} which converge at precisely polynomial rates; Carlotto-Chodosh-Rubinstein \cite{CarlottoChodoshRubinstein}, for example, have constructed such ``slowly-converging'' solutions for the Yamabe flow.)

The analogs of Theorem \ref{thm:decay} for parabolic equations are classical \cite{AgmonNirenberg}, \cite{CohenLees}. For the linear heat equation, there is 
a particularly elementary proof. Given a solution $u:M\times[0, \infty)\to \RR$ to $\pd{u}{t} = \Delta u$ on a compact manifold $(M, g)$, a short computation shows that the $L^2$-norm $E(t) = \int_M u^2(x, t)\, dV_g$ is log-convex in $t$. This implies  
that 
\[
            E(t) \geq E(0) e^{-N_0t},  \quad N_0 = \log(E(0)/E(1))
\]
for $t\geq 0$. The work which follows was originally motivated by a desire to find a comparably direct and effective argument for solutions of \eqref{eq:nrf}, where the situation is complicated by the nonlinearity of the equation and the degeneracy induced by its invariance under the diffeomorphism group. The problem is to parlay the sequential convergence in the hypotheses into smooth convergence relative to some gauge in which the rate of convergence can be conveniently measured and is comparable to the rate of convergence of the original solution.

Our approach is to reduce the problem to that of smooth convergence under the gradient flow of an entropy functional $\mu_{\sigma}$ which encodes normalized, fixed-scale versions of
Perelman's $\mu$- and $\lambda$-functionals and the expander entropy $\mu_{+}$ introduced in \cite{FeldmanIlmanenNi}. Our argument in fact shows that under the assumptions
of Theorem \ref{thm:decay}, for $a \gg 0$, there are (reasonably effective) constants $C_0$ and $N_0$ such that
\[
        \mu_{\sigma}(\bar{g}) - \mu_{\sigma}(g(t)) \geq C_0\|\nabla\mu_{\sigma}g(a)\|^2_{g(a)}e^{-N_0t}
\]
for all $t\geq a$. Here $\|\cdot\|_{g(a)}$ denotes a weighted $L^2$-norm; we will give precise definitions shortly for $\mu_{\sigma}$ and $\|\cdot\|_{g(a)}$ in Sections \ref{sec:modflow} and \ref{sec:evcomp} below.

For the reduction to the gradient flow, we use the \L{}ojasiewicz-Simon inequalities and associated methods developed by Sun-Wang \cite{SunWang} (cf. \cite{Ache}) for the case $\sigma = -1$, Haslhofer-M\"uller \cite{HaslhoferMueller} (cf. \cite{Haslhofer}) for the case $\sigma = 0$, and Kr\"oncke \cite{Kroencke} for the case $\sigma = 1$. The approaches in \cite{Ache, HaslhoferMueller, Kroencke} in turn rely on the general \L{}ojasiewicz-Simon inequality in \cite{ColdingMinicozzi}.

To estimate the rate of convergence of these gradient flows, we derive a differential inequality for the Dirichlet quotient associated to the weighted Einstein operator along the flow, modifying the classical technique of Agmon-Nirenberg \cite{AgmonNirenberg} to fit our situation. This portion of the argument relies on an analysis of the entropy functional $\mu_{\sigma}$ at the level of its 
third variation. From an upper bound on this Dirichlet-Einstein quotient, we obtain a lower bound on the weighted $L^2$-norm of the gradient $\nabla \mu_{\sigma}$ of $\mu_{\sigma}$ along the flow
and a corresponding bound for the entropy $\mu_{\sigma}$ itself. When $(M, \bar{g})$ is Einstein and one assumes smooth convergence in place of sequential convergence
in Theorem \ref{thm:decay}, one can avoid the passage through the modified flow and work directly with \eqref{eq:nrf}. See the remarks in Section \ref{ssec:einsteincase}.

 The conclusion of Theorem \ref{thm:decay} is related in some respects to the unique continuation result in \cite{KotschwarWangCylindrical}, which asserts that any shrinking soliton on an end $\Ec\subset S^k\times \RR^{n-k}$  which agrees to infinite order at spatial infinity with the standard cylindrical metric must coincide with the cylinder (or with a quotient thereof). The parabolic interpretation of this result is that any shrinking self-similar solution to the unnormalized Ricci flow
 defined on $\Ec \times [0, 1)$ which agrees to infinite order at spatial infinity and near the singular time $t=1$ with the shrinking cylindrical solution
 must be identical (up to a quotient) to that solution.  In the edge-case $k=n$ (which is not addressed in \cite{KotschwarWangCylindrical}), dropping the assumption of self-similarity, the question is whether a maximal solution $g(t)$ to the unnormalized Ricci flow $g(t)$ on $S^n\times [0, 1)$ which agrees to infinite order with the standard shrinking sphere
 $\mathring{g}(t) = 2(n-1)(1-t)g_{S^n}$ at the singular time (in the sense that
 $\|g(t) - \mathring{g}(t)\|_{C^k} = O((1-t)^m)$ for all 
 $m$ and all $k$ sufficiently large) must coincide with $\mathring{g}(t)$. Theorem \ref{thm:decay} implies in particular that it must.

In this direction, Strehlke \cite{StrehlkeUC} has recently proven (by an alternative approach) that a closed convex solution
to the normalized mean-curvature flow cannot converge faster than exponentially to the round sphere unless it coincides with the sphere. His result has applications to the problem of the regularity
of the arrival time function near an isolated critical point \cite{Strehlke2}.

\section{Entropy and the modified flow}\label{sec:modflow}
Let $M$ be a closed manifold and let $\mathcal{R}(M)$ denote the set of Riemannian metrics on $M$.
Define $\Wc_{\sigma}:\mathcal{R}(M) \times C^{\infty}(M)  \longrightarrow \RR$
by
\[
 \Wc_{\sigma}(g, f) = \int_M\left(R + |\nabla f|^2 -\sigma f \right)e^{-f}\,dV_g.
\]
When $\sigma = -1$, $\Wc_{\sigma}(g, f)$ is a normalized version of Perelman's entropy $\Wc(g, f, 1)$ \cite{Perelman1} 
taken at the fixed scale $\tau =1$, and, when $\sigma= 1$, it is a normalized version of the analogous expander entropy $\Wc_{+}(g, f, 1)$
introduced in \cite{FeldmanIlmanenNi}. When $\sigma = 0$, $\Wc_{\sigma}(g, f)$ is Perelman's $\mathcal{F}$-energy. 

As in \cite{Perelman1}, \cite{FeldmanIlmanenNi}, we define
\[  
 \mu_{\sigma}(g) = \inf \bigg\{\,\Wc_{\sigma}(g, f)\,\bigg|\,f\in C^{\infty}(M), \ \int_M e^{-f}\,dV_g = 1\,  \bigg\}.
\]
For fixed $g$, $\mu_{\sigma}(g)$ is finite and is achieved by a smooth minimizer $f$  satisfying the equation
\begin{equation}\label{eq:feq}
 2\Delta f - |\nabla f|^2 + R -\sigma f  = \mu_{\sigma}(g).
\end{equation}
Note that $\mu_0$ is Perelman's $\lambda$-functional, while $\mu_{-1}$ and $\mu_1$ are, respectively, normalized versions 
of Perelman's $\mu$-entropy \cite{Perelman1} and its expanding counterpart $\mu_{+}$ defined in \cite{FeldmanIlmanenNi}.
As a map $\mu_{\sigma}:\mathcal{R}(M)\longrightarrow \RR$, $\mu_{\sigma}$ is diffeomorphism-invariant. Along a solution $g(t)$ to \eqref{eq:nrf},
$\mu_{\sigma}(g(t))$ is monotone-increasing
and is constant precisely when $g(t)$ is self-similar. 

It is shown in \cite{HaslhoferMueller, Kroencke, SunWang} that, for metrics $g$ sufficiently close to a soliton $\bar{g}$, the minimizer $f = f_g$ is unique and depends analytically on $g$. (The minimizer $f_g$ is always unique when $\sigma = 0$, $1$, but need not be if $\sigma = -1$.)  At such metrics, the $L^2(e^{-f}dV_g)$-gradient of $\mu_{\sigma}$ 
is given by
\[
 \nabla \mu_{\sigma}(g) = - \left(\Rc(g) + \nabla\nabla f_{g} + \frac{\sigma}{2}g\right).
\]

We wish to convert Theorem \ref{thm:decay} into a question of the maximal rate of smooth convergence under the gradient flow of $\mu_{\sigma}$, that is,
 \begin{equation}\label{eq:mrf}
     \pd{g}{t} = -2\left(\Rc(g) +\nabla\nabla f_{g} + \frac{\sigma}{2}g\right).
 \end{equation}
We will call \eqref{eq:mrf} the \emph{modified Ricci flow}. Along a solution to \eqref{eq:mrf},
\begin{equation}\label{eq:muev}
 \frac{d}{dt}\mu_{\sigma}(g(t)) = 2\int_M \left| \Rc(g(t)) + \nabla\nabla f_{g(t)} + \frac{\sigma}{2} g(t)\right|^2_{g(t)}e^{-f_{g(t)}}\,dV_{g(t)}
\end{equation}

The well-posedness and long-time existence of this equation are troublesome matters in general, but for our purposes, we will need only to work in the vicinity of a soliton, where a solution to \eqref{eq:nrf} can be transformed into a local solution to \eqref{eq:mrf} by pull-back by a family of diffeomorphisms.
The facts we need are contained in
Lemmas 2.2 and Lemma 3.1 of \cite{SunWang} (cf. \cite{Ache}) for the case $\sigma =-1$, Theorem 3 of \cite{HaslhoferMueller} for the case $\sigma = 0$, and Lemma 6.1 and Theorem 6.2 of \cite{Kroencke} for the case $\sigma = 1$. We collect these facts in the theorem below. 
\begin{theorem}[\cite{HaslhoferMueller, Kroencke, SunWang}]
\label{thm:ls}
 Let $k\geq 2$ and $\alpha \in (0, 1)$. Suppose $(M, \bar{g})$ satisfies \eqref{eq:grs}. Then there is a $C_{\bar{g}}^{k, \alpha}$ neighborhood $U$ of $\bar{g}$ 
 in $\mathcal{R}(M)$
 such that, for all $g\in U$:
 \begin{enumerate}
  \item[(a)] There is a unique $f_g\in C^{\infty}(M)$ satisfying 
  \[
    \mu_{\sigma}(g) = \Wc_{\sigma}(g, f_g), \quad \int_M e^{-f_g}\,dV_g = 1,
\]
and the map $P: U \to C^{k, \alpha}(M)$ with $P(g) = f_g$ is analytic.

\item[(b)] There are constants $C_L$ and $\theta \in [1/2, 1)$ such that
  \begin{equation}\label{eq:lsineq1}
    |\mu_{\sigma}(\bar{g}) - \mu_{\sigma}(g)|^{\theta} \leq C_L \|\Rc(g)+ \nabla\nabla f_g + (\sigma/2)g\|_{L^2(e^{-f_g}dV_g)}.
  \end{equation}
 \end{enumerate}
\end{theorem}
In the shrinking case ($\sigma=-1$), the range of $k$ stated above differs slightly from that in \cite{SunWang}. To 
obtain the above statement for $k\geq 2$ (as opposed to $k \gg 1$) one may replace the implicit function theorem argument in Lemma 2.2 in \cite{SunWang} with an the analog of the argument in Lemma 6.1 of \cite{Kroencke}, which is based on a second-order rather than fourth-order operator, and use the alternative proof of the \L{}ojasiewicz inequality via \cite{ColdingMinicozzi}  in Appendix A of the paper of Ache \cite{Ache}. The improvement in regularity is not essential for our purposes, however.

Following \cite{Ache, SunWang}, we will call neighborhoods $U$ of the kind guaranteed by the theorem \emph{regular} neighborhoods and metrics \emph{regular}
if they belong to a regular neighborhood.

\section{Some preliminaries on the modified flow}\label{sec:evcomp}
In this section, we will derive some fundamental identities for solutions to \eqref{eq:mrf}. We will first need to fix some notation.
\subsection{Identities on weighted Riemannian manifolds} Recall that, for $f\in C^{\infty}(M)$, 
the weighted divergence operator $\dv_f: C^{\infty}(S^2T^*M) \to C^{\infty}(T^*M)$  and its $L^2(e^{-f}dV_g)$-adjoint $\dv_f^*: C^{\infty}(T^*M) \to C^{\infty}(S^2T^*M)$ are defined by
\begin{align*}
    \dv_f W &= e^{f}\dv(e^{-f}W) = \dv W - W(\nabla f, \cdot), \quad (\dv_f^*X) = \dv^* X = -\Lc_{X^{\sharp}}g,
\end{align*}
that is,
\begin{align*}
   (\dv_f W)_j = \nabla_i W_{ij} - W_{ij}\nabla_i f, \quad  (\dv_f^*X)_{ij} = -\frac{1}{2}(\nabla_i X_j + \nabla_j X_i),
\end{align*}
and the $f$-Laplacian and $f$-Einstein operators
\[
\Delta_f: C^{\infty}(T^k(T^*M)) \to C^{\infty}(T^k(T^*M)), \quad \Box_f: C^{\infty}(S^2T^*M)\to C^{\infty}(S^2T^*M),
\]
act on smooth tensor fields $W$ by
\[
    \Delta_f W = \Delta W - \nabla_{\nabla f}W, \quad \Box_f W = \Delta_f W + 2\Rm(W).
\]
Here, $\Rm(W)_{ij} = R_{iklj}W_{kl}$; we will also at times write $\Rm(U, W) = \langle \Rm(U), W\rangle$. 
We further define 
\begin{align}
    \label{eq:sdef}
    S^{\sigma}(g, f) &= \Rc(g) + \nabla\nabla f + \frac{\sigma}{2}g, \\
    \label{eq:mdef}
    M^{\sigma}(g, f) &=  2\Delta f - |\nabla f|^2 + R -\sigma f,
\end{align}
for $g\in \mathcal{R}(M)$. 

When $g$ is a soliton with potential $f$ satisfying \eqref{eq:grs}, $S^{\sigma}(g, f) = 0$ and $M^{\sigma}(g, f) = \mathrm{const}$. When $f$ is an entropy minimizer for $g$, satisfying $\int_M e^{-f}\,dV_g = 1$
and $\Wc_{\sigma}(g, f) = \mu_{\sigma}(g)$,  we have $\nabla \mu_{\sigma}(g) = -S^{\sigma}(g, f)$ and $M^{\sigma}(g, f) = \mu_{\sigma}(g)$.
In general, $S^{\sigma}$ and $M^{\sigma}$ are related by the following weighted Bianchi-type identity. 
\begin{lemma}\label{lem:divfv}
For all $g\in \mathcal{R}(M)$ and $f\in C^{\infty}(M)$,
\begin{equation}\label{eq:divfv}
    \dv_fS^{\sigma}(g, f) = \frac{1}{2}\nabla M^{\sigma}(g, f).
\end{equation}
\end{lemma}
\begin{proof}
Note that
\[
    (\dv_f \Rc)_j = \nabla_i R_{ij} - \nabla_i f R_{ij} = \frac{\nabla_jR}{2}  - \nabla_i f R_{ij},
\]
and
\begin{align*}
 (\dv_f \nabla\nabla f)_j &= \Delta \nabla_j f - \nabla_i f\nabla_i\nabla_j f = \nabla_i f R_{ij} + \nabla_j\left(\Delta f - \frac{|\nabla f|^2}{2}\right),
\end{align*}
while
\begin{equation*}
 \dv_f((\sigma/2) g)_j = -(\sigma/2)\nabla_j f.
\end{equation*}
Summing these identities yields \eqref{eq:divfv}. 
\end{proof}

The following commutator identity is key to the estimate of the Dirichlet quotient in the next section. The proof is a straightforward but somewhat lengthy calculation.
\begin{lemma}\label{lem:commidentity}
 Suppose $f\in C^{\infty}(M)$. Then $S^{\sigma} = S^{\sigma}(g, f)$
satisfies
 \begin{align}\label{eq:commidentity}
  \begin{split}
  \dv_f \Box_f S^{\sigma} &= \Delta_f \dv_f S^{\sigma} -\frac{\sigma}{2}\dv_f S^{\sigma} + \frac{1}{2}\nabla|S^{\sigma}|^2 + \frac{1}{2}S^{\sigma}(\nabla M^{\sigma}, \cdot).
  \end{split}
 \end{align}
 In particular, 
 \begin{equation}\label{eq:specvident}
 \dv_{f} \Box_{f} S^{\sigma} = \frac{1}{2}\nabla|S^{\sigma}|^2
\end{equation}
if $M^{\sigma}(g, f)$ is constant.
\end{lemma}
\begin{proof}
 Write $S= S^{\sigma}$ and $M = M^{\sigma}$.
 First, we have
 \begin{align*}
    &(\dv_f \Delta_f S)_j = \nabla_i\left(\Delta S_{ij} - \nabla_k S_{ij}\nabla_k f\right) - (\Delta S_{ij} - \nabla_kS_{ij}\nabla_kf) \nabla_if\\
                         &\qquad\qquad=   \nabla_i \Delta S_{ij} - \nabla_i\nabla_k S_{ij}\nabla_k f - \nabla_k S_{ij}\nabla_i\nabla_k f - \Delta S_{ij} \nabla_i f + \nabla_kS_{ij}\nabla_i f\nabla_kf,
 \end{align*}
while
\begin{align*} 
 (\Delta_f \dv_f S)_j &= \Delta\left(\nabla_i S_{ij} - \nabla_i f S_{ij}\right) - \nabla_k f\nabla_k(\nabla_i S_{ij} - \nabla_i f S_{ij})\\
\begin{split}
                      &= \Delta \nabla_i S_{ij} - \Delta \nabla_i f S_{ij}  - 2\nabla_k\nabla_i f\nabla_k S_{ij} - \Delta S_{ij}\nabla_i f   -\nabla_k\nabla_i S_{ij} \nabla_k f\\
                      &\phantom{=}
                     + \frac{1}{2}\nabla_i |\nabla  f|^2 S_{ij} + \nabla_k S_{ij}\nabla_k f\nabla_i f \\
                      &= \Delta \nabla_i S_{ij} - \nabla_k\nabla_i S_{ij}\nabla_k f -S_{ij}\nabla_i \left(\Delta f - \frac{1}{2}|\nabla f|^2\right) - R_{ik}S_{ij}\nabla_k f\\
                      &\phantom{=}-2\nabla_k\nabla_i f \nabla_kS_{ij} - \Delta S_{ij} \nabla_i f + \nabla_k S_{ij} \nabla_i f\nabla_k f,
\end{split}
 \end{align*}
 so
 \begin{align*}
  [\dv_f, \Delta_f] S_j &= [\nabla_i, \Delta]S_{ij} -[\nabla_i, \nabla_k] S_{ij}\nabla_kf +\nabla_i\nabla_k f \nabla_k S_{ij} +R_{ik}\nabla_kf S_{ij}\\
                        &\phantom{=} + S_{ij}\nabla_i \left(\Delta f - \frac{1}{2}|\nabla f|^2\right)\\                        
&= [\nabla_i, \Delta ]S_{ij} +R_{ikjp}S_{ip}\nabla_kf +\nabla_i\nabla_k f \nabla_k S_{ij}\\
    &\phantom{=}+ S_{ij}\nabla_i \left(\Delta f - \frac{1}{2}|\nabla f|^2\right),                        
 \end{align*}
where we have used
\[
    -[\nabla_i, \nabla_k] S_{ij} = R_{ikip}S_{pj} + R_{ikjp}S_{ip} = -R_{kp}S_{pj} + R_{ikjp}S_{ip}.
\]

Now,
\begin{align*}
[ \nabla_i, \Delta] S_{ij} = \frac{1}{2}\nabla_pR S_{pj} + S_{ip}(\nabla_p R_{ji} - \nabla_j R_{pi}) + R_{ap}\nabla_a S_{pj} - 2R_{iajp}\nabla_a S_{ip},
\end{align*}
which, since
\begin{align*}
   S_{ip}(\nabla_p R_{ji} - \nabla_j R_{pi}) &= S_{ip}(\nabla_p S_{ji} - \nabla_j S_{pi}) + S_{ip}R_{pjiq}\nabla_q f,
\end{align*}
can be rewritten as 
\begin{align*}
[\nabla_i, \Delta] S_{ij} &= \frac{1}{2}\nabla_pR S_{pj} + S_{ip}(\nabla_p S_{ji} - \nabla_j S_{pi}) + S_{ip}R_{pjiq}\nabla_q f + R_{ap}\nabla_a S_{pj}\\
    &\phantom{=}- 2R_{iajp}\nabla_a S_{ip}.
\end{align*}
Combined with the expression above, then, we have
\begin{align}
 \begin{split}\label{eq:comm1}
      [\dv_f, \Delta_f] S_j &=  S_{ip}(\nabla_p S_{ji} - \nabla_j S_{pi}) +  2R_{iajp}\nabla_a S_{ip} +  R_{ap}\nabla_a S_{pj}\\
      &\phantom{=}
                        +\nabla_i\nabla_k f \nabla_k S_{ij} + \frac{1}{2}S_{ij}\nabla_i \left(2\Delta f - |\nabla f|^2 + R\right)\\
                         &=  S_{ip}(2\nabla_p S_{ji} - \nabla_j S_{pi}) +  2R_{iajp}\nabla_a S_{ip}  -\sigma \dv_f S_{j}.\\ 
      &\phantom{=}
                        + \frac{1}{2}S_{ij}\nabla_i \left(2\Delta f - |\nabla f|^2 + R - \sigma f\right).
 \end{split}
\end{align}

On the other hand,
\begin{align*}
    \dv_f \Rm(S)_j &= \nabla_i R_{ipqj} S_{pq} + R_{ipqj}\nabla_i S_{pq} - \nabla_i f R_{ipqj}S_{pq}\\
                    &= (\nabla_j R_{qp} - \nabla_q R_{jp} - \nabla_i f R_{jqpi})S_{pq} +  R_{ipqj}\nabla_i S_{pq}\\
                    &= (\nabla_j S_{qp} - \nabla_q S_{jp}) S_{pq}  +  R_{ipqj}\nabla_i S_{pq},
\end{align*}
which, combined with  \eqref{eq:comm1}, yields
\begin{align*}
     &(\dv_f \Box_f S - \Delta_f \dv_f) S_j  =  [\dv_f, \Delta_f] S_j + 2\dv_f(\Rm(S))_j \\
     &\qquad\qquad=  [\dv_f, \Delta_f] S_j + 2(\nabla_j S_{qp} - \nabla_q S_{jp}) S_{pq}  +  2R_{ipqj}\nabla_i S_{pq}\\
     &\qquad\qquad = S_{ip}\nabla_j S_{ip} -\sigma \dv_f S_{j} + \frac{1}{2}S_{ij}\nabla_iM\\
     &\qquad\qquad= \frac{1}{2} \nabla_j |S|^2 -\sigma \dv_f S_j + \frac{1}{2}S_{ij}\nabla_iM,
\end{align*}
which is \eqref{eq:commidentity}.
\end{proof}

\subsection{Evolution equations under the modified Ricci flow.}
For the remainder of this section, we will assume that $g = g(t)$ is a smooth solution to the modified Ricci flow \eqref{eq:mrf} on $M\times [0, T)$, $T\in (0, \infty]$, 
with associated potential $f = f_{g(t)}$ satisfying 
\begin{equation}\label{eq:ftmin}
        \Wc_{\sigma}(g(t), f_{g(t)}) = \mu_{\sigma}(g(t)), \quad \int_M e^{-f_{g(t)}}\,dV_{g(t)} = 1,
\end{equation}
and therefore also 
\[
M^{\sigma}(g(t), f_{g(t)}) = \mu_{\sigma}(g(t)),
\]
for $t\in [0, T)$. Since $f$ is determined
by $g$, we will write $S^{\sigma}(g)$ (or simply $S^{\sigma}$) in place of $S^{\sigma}(g, f)$.

In the computations below, it will be convenient to use
the operator $D_t$ which acts on families of $k$-tensors $V = V(t)$ by
\[
 D_t V_{a_1 a_2 \cdots a_k} = \pdt V_{a_1a_2\cdots a_k} - S^{\sigma}_{a_1 p} V_{p a_2 \cdots a_k} - S^{\sigma}_{a_2p} V_{a_1p \cdots a_k} 
 - \cdots - S^{\sigma}_{a_kp} V_{a_1a_2 \cdots p},
\]
which satisfies the product rule
\[
    \frac{d}{dt} \langle V, W\rangle = \langle D_t V, W\rangle + \langle V, D_t W\rangle,  
\]
relative to the inner product on $T^k(T^*M)$ induced by $g(t)$. 
In terms of a smooth family $\{e_i(t)\}_{i=1}^n$ of local frames which is evolving so as to remain orthogonal under \eqref{eq:mrf},
$D_{t} V_{a_1 a_2 \cdots a_k}$ expresses the components of the total derivative
\[
    D_t V_{a_1a_2 \cdots a_k} = \frac{d}{dt}V(e_{a_1}(t), e_{a_2}(t), \ldots, e_{a_k}(t)).
\]
See \cite{HamiltonHarnack} or Appendix F of \cite{RFV2P2}
for a natural interpretation of the operator $D_t$.

Under \eqref{eq:mrf}, we have $\pdt\left(e^{-f}dV_g\right) = -H dV_g$
where
\begin{equation}\label{eq:hdef}
    H = \pd{f}{t} + \operatorname{tr}_{g} S^{\sigma}.
\end{equation}
Together, $S^{\sigma}$ and $H$ satisfy a coupled system of equations which we now derive.

\begin{proposition}\label{prop:sev}
 Suppose that $g = g(t)$ satisfies \eqref{eq:mrf} on $M\times [0, T)$ and $f = f_{g(t)}$. Then $S^{\sigma} = S^{\sigma}(g(t))$ satisfies
 \begin{equation}\label{eq:sev}
 D_t S^{\sigma} = \Box_f S^{\sigma} + \nabla\nabla H.
 \end{equation}
\end{proposition}

\begin{proof} We first compute the linearization of the operator 
\[
S^{\sigma}(g, f) = \Rc(g) + \nabla\nabla f + (\sigma/2)g
\]
acting on general $g\in \mathcal{R}(M)$ and $f\in C^{\infty}(M)$ in the directions $\delta g = h$ and $\delta f = k$. We claim that
\begin{align}\label{eq:slin}
 \begin{split}
    \delta S^{\sigma}_{ij} &= - \frac{1}{2}\Delta_f h_{ij} - \Rm(h)_{ij} - \dv_f^*\dv_f h_{ij}
    + \frac{1}{2}\left(S^{\sigma}_{il}h_{lj} + h_{il}S^{\sigma}_{lj}\right)\\
    &\phantom{=} + \nabla_i\nabla_j\left(k - \frac{1}{2}\operatorname{tr}_gh\right).
 \end{split}
 \end{align}

 To see this, we start with the well-known formula
\begin{align*}
\begin{split}
    \delta R_{ij} &= -\frac{1}{2}\Delta h_{ij} - \Rm(h)_{ij} -\dv^*\dv h_{ij} - \frac{1}{2}\nabla_i\nabla_j \operatorname{tr}_gh  +\frac{1}{2}(R_{il}h_{lj} + R_{jl} h_{il})
\end{split}
\end{align*}
for the linearization of the Ricci tensor (see, e.g., \cite{Besse}).
Using
\[
\dv h_j = \dv_f h_j + h(\nabla f)_j, \quad \dv_f^* = \dv^*,
\]
we obtain
\begin{align*}
 \dv^*\dv h_{ij} &= \dv^*\dv_f h_{ij} + \dv^*(h(\nabla f))_{ij}\\
    &=\dv_f^*\dv_f h_{ij} -\frac{1}{2}( h_{il}\nabla_l \nabla_j f + h_{jl}\nabla_i\nabla_l f) -\frac{1}{2}(\nabla_i h_{jl}+ \nabla_j h_{il})\nabla_l f,
\end{align*}
and thus
\begin{align*}\label{eq:rclin2}
\begin{split}
   \delta R_{ij} &= -\frac{1}{2}\Delta_f h_{ij} - \Rm(h)_{ij} -\dv_f^*\dv_f h_{ij} - \frac{1}{2}\nabla_i\nabla_j \operatorname{tr}_gh\\
   &\phantom{=}
   +\frac{1}{2}(R_{il}+\nabla_i\nabla_lf)h_{lj} + \frac{1}{2}(R_{jl} + \nabla_j\nabla_l f) h_{il}\\
   &\phantom{=}
   +\frac{1}{2}\left(\nabla_ih_{jl} + \nabla_j h_{il} - \nabla_l h_{ij}\right)\nabla_l f.
\end{split}
\end{align*}
But,
\begin{align*}
 \begin{split}
    \delta \nabla_i\nabla_j f = \nabla_i \nabla_j k - \frac{1}{2}\left(\nabla_ih_{jl} + \nabla_j h_{il} - \nabla_l h_{ij}\right)\nabla_l f,
 \end{split}
 \end{align*}
and by summing the identities, we arrive at
\begin{align*}
 \delta (R_{ij} + \nabla_i\nabla_j f) &= -\frac{1}{2}\Delta_f h_{ij} - \Rm(h)_{ij} -\dv_f^*\dv_f h_{ij} +\nabla_i\nabla_j\left(k -  \frac{1}{2}\operatorname{tr}_gh\right)\\
   &\phantom{=}
   +\frac{1}{2}(R_{il}+\nabla_i\nabla_lf)h_{lj} + \frac{1}{2}(R_{jl} + \nabla_j\nabla_l f) h_{il}\\
   &= -\frac{1}{2}\Delta_f h_{ij} - \Rm(h)_{ij} -\dv_f^*\dv_f h_{ij} +\nabla_i\nabla_j\left(k -  \frac{1}{2}\operatorname{tr}_gh\right)\\
   &\phantom{=}
   +\frac{1}{2}\left(S^{\sigma}_{il}h_{lj} + S^{\sigma}_{jl} h_{il}\right) - \frac{\sigma}{2} h_{ij},
\end{align*}
and \eqref{eq:slin} follows.

Now we turn to the evolution equation \eqref{eq:sev} for $S = S^{\sigma}(g(t))$. By Lemma \ref{lem:divfv}, $\dv_f S = 0$ for each $t$, so applying \eqref{eq:slin} with $h_{ij} = -2S_{ij}$
and $k = \pd{f}{t}$, we find that
\begin{align*}
    \pdt S_{ij} &= \Delta_f S_{ij} + 2\Rm(S)_{ij} - 2S_{ik}S_{kj} + \nabla_i\nabla_j\left(\pd{f}{t} + \operatorname{tr}_gS\right).
\end{align*}
Since
\[
 D_t S_{ij} = \pdt S_{ij} + 2S_{ik}S_{kj}, 
\]
and
\[
\pd{f}{t} + \operatorname{tr}_gS = \pd{f}{t} + \Delta f + R + \frac{\sigma n}{2} = H,
\]
we then obtain \eqref{eq:sev}.
\end{proof}

Here and elsewhere,
we will use $(\cdot, \cdot)$ to denote the $L^2$-inner product
\[
 (V, W) =  \int_M \langle V, W\rangle_g\, e^{-f} dV_g
\]
on $C^{\infty}(T^k(T^*M))$ and $\|\cdot\|$ to denote the associated norm.

\begin{proposition}\label{prop:heq}
 Suppose that $g = g(t)$ satisfies \eqref{eq:mrf} on $M\times [0, T)$ and $f = f_{g(t)}$.
 Then
 \begin{equation}\label{eq:heq}
    \Delta_f H = \frac{\sigma}{2}H - \left|S^{\sigma}\right|^2 + \|S^{\sigma}\|^2, \quad \int_M H\,e^{-f}\,dV_g \equiv 0. 
 \end{equation}
\end{proposition}
\begin{proof}
The second equation in \eqref{eq:heq} follows directly from \eqref{eq:ftmin} and \eqref{eq:hdef}.
We will derive the first equation from \eqref{eq:feq}.

Let $S = S^{\sigma}$. Using $\dv_f S = \dv S - S(\nabla f, \cdot) = 0$, we compute that
\begin{align}
\begin{split}\label{eq:ddtlap} 
  \pdt\left(\Delta f\right) 
  &=\Delta \pd{f}{t} + 2\langle S, \nabla\nabla f\rangle  +\langle 2\dv S - \nabla \operatorname{tr}_gS, \nabla f\rangle\\
  &=\Delta \pd{f}{t} + 2\langle S, \nabla\nabla f\rangle  + 2S(\nabla f, \nabla f) - \langle \nabla \operatorname{tr}_gS, \nabla f\rangle.
\end{split}
\end{align}
Next, we see that
\begin{align}
\begin{split}\label{eq:ddtgrad}
    \pdt|\nabla f|^2 &= 2S(\nabla f, \nabla f) + 2 \left\langle \nabla\pd{f}{t}, \nabla f \right\rangle, 
\end{split} 
\end{align}
and, using 
\[
    \dv\dv S = \dv(\dv_f S) + \langle \dv S, \nabla f\rangle + \langle S, \nabla\nabla f\rangle
    = S(\nabla f, \nabla f) + \langle S, \nabla\nabla f\rangle,
\]
together with standard variation formulas for the scalar curvature, that
\begin{align}
\begin{split}\label{eq:ddtR}
    \pdt R &= 2\Delta \operatorname{tr}_gS - 2\dv(\dv S) + 2 \langle S, \Rc\rangle\\
           &= 2\Delta \operatorname{tr}_gS - 2S(\nabla f,\nabla f)  + 2\langle S, \Rc - \nabla \nabla f\rangle.
\end{split}
\end{align}
From \eqref{eq:ddtlap}, \eqref{eq:ddtgrad}, and \eqref{eq:ddtR}, we then obtain
\begin{align*}
\frac{d}{dt}\mu_{\sigma}(g(t)) &= \pdt\left(2\Delta f - |\nabla f|^2 + R -\sigma f\right)\\
    &=  2 \Delta H - 2\left\langle \nabla H, \nabla f\right\rangle + 2\left \langle S, \Rc + \nabla\nabla f\right\rangle -\sigma \pd{f}{t}\\
    &= 2\Delta_f H - \sigma H + 2|S|^2.
\end{align*}
Since $\frac{d}{dt} \mu_{\sigma}(g(t)) = 2\|S\|^2$, this implies the first equation in \eqref{eq:heq}.
\end{proof}

\section{Estimating the Dirichlet-Einstein quotient} 
In this section, we continue to assume that $g = g(t)$ is a smooth solution to the modified Ricci flow \eqref{eq:mrf} on $M\times [0, T)$. Define
\[
  E(t) = \|S^{\sigma}\|^2, \quad F(t) = \|\nabla S^{\sigma}\|^2 - 2\left(\Rm(S^{\sigma}), S^{\sigma}\right),
\]
where $S^{\sigma} = S^{\sigma}(g(t))$ and, as before, $\|\cdot\|$ and $(\cdot, \cdot)$ denote the $L^2$-norm and inner product relative to $g(t)$ and $e^{-f_{g(t)}}\,dV_{g(t)}$. When $E(t) \neq 0$, we define the Dirichlet-Einstein quotient 
\begin{equation}\label{eq:dequotientdef}
N(t) = F(t)/ E(t).
\end{equation}
We will derive a differential inequality for $N$ which will allow us to control it from above. The upper bound on $N$ will
in turn yield a lower bound on $E$.
Define the operators $\Lc_0$ and $\Lc_1$ by
\[
 \Lc_{0}W = D_tW - \Box_fW - \nabla\nabla H, \quad \Lc_{1}W = D_tW + \Box_fW - \nabla\nabla H.
\]
By \eqref{eq:sev}, $\Lc_0 S^{\sigma} \equiv 0$. 

\subsection{Evolution equations}
In the next two propositions, we compute the evolution equations for the quantities $E(t)$ and $F(t)$.
\begin{proposition}\label{prop:eev}
Suppose that $g(t)$ solves \eqref{eq:mrf} on $M\times [0, T)$ and $f = f_{g(t)}$. Then $E$ satisfies
 \begin{align}
 \label{eq:eev1}
 \begin{split}
    \dot{E}(t) &= -2 F(t) - (HS^{\sigma}, S^{\sigma})\\
    &= (\Lc_{1}S^{\sigma}, S^{\sigma}) -(HS^{\sigma}, S^{\sigma}).
 \end{split}
 \end{align}
\end{proposition}
\begin{proof} Let $S = S^{\sigma}$. To begin, we have
 \begin{equation}\label{eq:e1}
    \dot{E}(t) = 2(D_t S, S) - (HS, S),
 \end{equation}
and using \eqref{eq:sev} and integrating by parts, we obtain 
\begin{align*}
  (D_tS, S) &= (\Delta_f S + 2\Rm(S) + 2\dv_f^*\nabla H, S)\\
            &= -\|\nabla S\|^2 + 2(\Rm(S), S) + 2(\nabla H, \dv_f S).
\end{align*}
But $\dv_f S \equiv 0$, so $(D_t S, S) = -F$, and the first identity follows from \eqref{eq:e1}.

For the second identity, we write the first term of \eqref{eq:e1} instead as
\begin{align*}
    2(D_t S, S) &= ((D_t +\Box_f) S, S) + ((D_t - \Box_f)S, S)\\
    &= (\Lc_1 S, S) + (\Lc_0 S, S) + 2(\nabla\nabla H, S)\\
    &= (\Lc_1 S, S), 
\end{align*}
since $\Lc_0 S = 0$, and $(\nabla\nabla H, S) = -(\nabla H, \dv_f S)  = 0$ as above.
\end{proof}

\begin{proposition}\label{prop:fev}
 Suppose that $g(t)$ solves \eqref{eq:mrf} on $M\times [0, T)$. Then $F$ satisfies
 \begin{equation}\label{eq:fid}
    F(t) = -\frac{1}{2}(\Lc_1 S^{\sigma}, S^{\sigma}), \quad \dot{F}(t) = -\frac{1}{2}\|\Lc_1 S^{\sigma}\|^2 -E^2(t) +  (|S^{\sigma}|^2 S^{\sigma} ,S^{\sigma}) +  J(t),
\end{equation}
where
\begin{align}\label{eq:jdef}
\begin{split}
J(t) &= 2\int_M\big(\left\langle [D_t, \nabla_i] S^{\sigma}, \nabla_i S^{\sigma}\right\rangle - (D_t\Rm)(S^{\sigma}, S^{\sigma})\big)e^{-f}\,dV\\
    &\phantom{=} - \int_M H\left(\frac{\sigma}{2}|S^{\sigma}|^2+|\nabla S^{\sigma}|^2 - 2\Rm(S^{\sigma}, S^{\sigma})\right)e^{-f}\,dV.
\end{split}
\end{align}
\end{proposition}
\begin{proof} Let $S = S^{\sigma}$.
 Integrating by parts and using that $\Lc_0 S = 0$ and $(\nabla\nabla H, S) = 0$, we see that
\begin{align*}
 \begin{split}
  F &= -(\Box_f S, S) =\frac{1}{2}\left(\big((D_t - \Box_f)S - \nabla\nabla H, S\big) - \big((D_t + \Box_f)S - \nabla\nabla H, S\big)\right)\\
  &= -\frac{1}{2}\big(\Lc_1 S, S),
 \end{split}
 \end{align*}
which is the first identity in \eqref{eq:fid}.

For the second identity,  we first compute directly that
\begin{align}\label{eq:f1}
 \begin{split}
  \dot{F} &= 2\int_M\big(\langle D_t \nabla_i S, \nabla_i S\rangle -2\Rm(D_t S, S)\big)e^{-f}\,dV\\
  &\phantom{=} - \int_M\big(H(|\nabla S|^2 - 2\Rm(S, S)) + 2D_t\Rm(S, S) \big)e^{-f}\,dV,
 \end{split}
\end{align}
and, then, that
\begin{align*}
  &\int_M\big(\langle D_t \nabla_i S, \nabla_i S\rangle -2\Rm(D_t S, S)\big)e^{-f}\,dV\\
  &\qquad=  \int_M\big(\langle [D_t,  \nabla_i] S, \nabla_i S\rangle - \langle \Delta_f S + 2\Rm(S), D_t S\rangle \big)e^{-f}\,dV\\
  &\qquad= \int_M\big(\langle [D_t,  \nabla_i] S, \nabla_i S\rangle - \langle \Box_f S, D_t S\rangle \big)e^{-f}\,dV.
\end{align*}
Next, we write
\begin{align*}
-(\Box_fS, D_t S) &= -(\Box_f S, D_t S - \nabla\nabla H) - (\Box_f S, \nabla\nabla H)\\
&= \frac{1}{4}\|\Lc_0 S\|^2 - \frac{1}{4}\|\Lc_1 S\|^2  + (\dv_f\Box_f S, \nabla H)\\
&= -\frac{1}{4}\|\Lc_1 S\|^2 + (\dv_f\Box_f S, \nabla H),
\end{align*}
and use \eqref{eq:commidentity} and \eqref{eq:heq} to compute that
\begin{align*}
  (\dv_f\Box_f S, \nabla H) &=  \frac{1}{2}(\nabla |S|^2, \nabla H) = -\frac{1}{2}(\Delta_f HS, S)\\
        &= \frac{1}{2}\left((|S|^2S, S) - \frac{\sigma}{2}(HS, S) - \|S\|^4\right).
\end{align*}
Thus,
\begin{align*}
&2\int_M\big(\langle D_t \nabla_i S, \nabla_i S\rangle -2\Rm(D_t S, S)\big)e^{-f}\,dV\\
&\qquad= - \frac{1}{2}\|\Lc_1 S\|^2 - \frac{\sigma}{2}(HS, S) +  (|S|^2S, S) - E^2,
\end{align*}
which, substituted into \eqref{eq:f1}, yields the identity for $\dot{F}$ in \eqref{eq:fid}.
\end{proof}

\subsection{Estimating error terms}
Next we estimate the error terms in the quantity $J(t)$ appearing in the evolution equation for $F(t)$.  We will use the notation 
\[
\|\cdot\|_{\infty}= \|\cdot\|_{\infty; g(t)} 
\]
to denote the supremum norm on the tensor bundles $T^k(T^*M)$ induced by $g(t)$.
\begin{proposition}\label{prop:jterms}
Suppose $g(t)$ is a solution to the modified Ricci flow \eqref{eq:mrf} on $M\times [0, T)$ with
\[
  \|\Rm(g(t))\|_{\infty} \leq K_0, \quad \|S^{\sigma}(g(t))\|_{\infty} \leq K_1.
\]
There is a constant $C = C(K_0, K_1)$ such that
\begin{equation}\label{eq:commest}
|([D_t, \nabla_i] S^{\sigma}, \nabla_i S^{\sigma})| \leq C\|S^{\sigma}\|_{\infty}(\|S^{\sigma}\|^2 + \|\nabla S^{\sigma}\|^2) 
\end{equation}
and
\begin{equation}\label{eq:dtrmest}
 |(D_t \Rm(S^{\sigma}), S^{\sigma})| \leq C\|S^{\sigma}\|_{\infty}(\|S^{\sigma}\|^2 + \|\nabla S^{\sigma}\|^2).
\end{equation}
Consequently, 
\begin{align}\label{eq:fest}
\begin{split}
    \dot{F}(t) &\leq -\frac{1}{2}\|\Lc_{1} S^{\sigma}\|^2  + C_0(\|H\|_{\infty} + \|S^{\sigma}\|_{\infty})(C_1E(t) + F(t))
\end{split}
\end{align}
for some constants $C_i = C_i(K_0, K_1)$.

\end{proposition}
\begin{proof} Let $S = S^{\sigma}$. The first inequality follows directly from the pointwise identity
\[
 [D_t, \nabla_i]S_{jk} = S_{il}\nabla_{l}S_{jk} + (\nabla_j S_{il} - \nabla_l S_{ij})S_{lk}
            + (\nabla_k S_{il} - \nabla_l S_{ik})S_{jl}.
\]
For the second inequality, we first rewrite the identity
\[
 \pdt R_{ijkl} = \nabla_i\nabla_l S_{jk} + \nabla_j\nabla_k S_{il}  - \nabla_i\nabla_k S_{jl}  - \nabla_j\nabla_l S_{ik}
               +R_{ijpl}S_{pk} + R_{ijkp} S_{pl}
\]
in terms of the $D_t$ operator to obtain
\begin{align}\label{eq:dtrm}
\begin{split}
 D_t R_{ijkl} &= \nabla_i\nabla_l S_{jk} + \nabla_j\nabla_k S_{il}  - \nabla_i\nabla_k S_{jl}  - \nabla_j\nabla_l S_{ik}\\
        &\phantom{=} -R_{pjkl}S_{ip} - R_{ipkl} S_{jp}.
\end{split}
\end{align}
Integrating by parts and using that $\dv_f S \equiv 0$, we then see that 
\begin{align*}
 \int_M \nabla_i\nabla_l S_{jk} S_{il}S_{jk} e^{-f}\,dV 
 &= -\int_M\nabla_l S_{jk} \left((\dv_f S)_lS_{jk}  + S_{il}\nabla_iS_{jk}\right)e^{-f}\,dV\\
 &=-\int_M\nabla_l S_{jk} S_{il}\nabla_iS_{jk}e^{-f}\,dV, 
\end{align*}
and hence that
\[
 \left|\int_M\nabla_i\nabla_l S_{jk} S_{il}S_{jk} e^{-f}\,dV\right| \leq C\|S\|_{\infty}\|\nabla S\|^2,
\]
for some universal constant $C$. Estimating terms two through four in \eqref{eq:dtrm} similarly,
we obtain that
\[
 |(D_t\Rm(S), S)|\leq C\|S\|_{\infty}(\|\nabla S\|^2 + K_0\|S\|^2).
\]

Finally, for \eqref{eq:fest}, note that
\[
        \|\nabla S\|^2 - 2K_0E \leq F  \leq \|\nabla S\|^2 + 2K_0E.
\]
Combining \eqref{eq:jdef} with \eqref{eq:commest} and \eqref{eq:dtrmest} we then have
\begin{align*}
    |J(t)| &\leq 2|( [D_t, \nabla_i] S, \nabla_iS)| + 2|(D_t\Rm(S), S)| \\
    &\phantom{\leq}+2\|H\|_{\infty}(\|\nabla S\|^2 + (|\sigma|/2)\|S\|^2  + |(\Rm(S), S)|)\\
    &\leq C_0(\|H\|_{\infty} + \|S\|_{\infty})(C_1E(t) + F(t)),
\end{align*}
for some $C_i = C_i(K_0)$.
Together with \eqref{eq:fid}, this yields
\begin{align*}
    \dot{F}(t) &= -\frac{1}{2}\|\Lc_1 S\|^2 -E^2(t) +  (|S|^2 S ,S) +  J(t)\\
               &\leq  -\frac{1}{2}\|\Lc_1 S\|^2 + C_0(\|H\|_{\infty} + \|S\|_{\infty} )(C_1E(t) + F(t))
    \end{align*}
for some $C_i = C_i(K_0, K_1)$.
\end{proof}

\subsection{A differential inequality for the Dirichlet-Einstein quotient}
Now we combine Propositions \ref{prop:eev}, \ref{prop:fev}, and \ref{prop:jterms} to derive the key differential inequality for $N(t)$ along the modified flow.
\begin{proposition}\label{prop:nest}
 Suppose $g(t)$ is a smooth solution to \eqref{eq:mrf} on $M\times [0, T)$, $T \in (0, \infty]$ for which $E(t)\neq 0$ and 
\[
  \|\Rm(g(t))\|_{\infty} \leq K_0, \quad \|S^{\sigma}(g(t))\|_{\infty} \leq K_1.
\] 
 Then there are constants $C_i = C_i(K_0, K_1)$, $i=0$, $1$ such that
 \begin{equation}\label{eq:nest}
    \dot{N}(t) \leq C_0\left(\|H(t)\|_{\infty} + \|S^{\sigma}(g(t))\|_{\infty} \right)(N(t) + C_1)
 \end{equation}
 for all $0 \leq t < T$, where $N(t)$ is as defined in \eqref{eq:dequotientdef}.
\end{proposition}
\begin{proof} Let $S = S^{\sigma}(g(t))$. We will use $C_i$ to denote a sequence of positive constants depending only on $K_0$ and $K_1$.
On one hand, by \eqref{eq:fest}, we have
\begin{align}
\label{eq:fdote}
 \dot{F}E &\leq -\frac{1}{2}\|\Lc_1 S\|^2\|S\|^2 +  C_0(\|H\|_{\infty}+ \|S\|_{\infty})E(C_1E+F)
\end{align}
On the other, using \eqref{eq:eev1} and \eqref{eq:fid} and that
\[
    -|F|\geq -\|\nabla S\|^2 - 2K_0 E \geq -F - 4K_0E,
\]
we have
\begin{align}
\nonumber
\dot{E}F &= -\frac{1}{2}(\Lc_1 S, S)^2 - (HS, S)F \\
\nonumber
&\geq -\frac{1}{2}(\Lc_1 S, S)^2 - \|H\|_{\infty}E(F + 4K_0E)\\
\label{eq:edotf}
&\geq -\frac{1}{2}(\Lc_1 S, S)^2 - C_0\|H\|_{\infty}E(C_1E+ F).
\end{align}

Combining \eqref{eq:fdote} and \eqref{eq:edotf}, we obtain
\begin{align*}
\dot{N} &= \frac{\dot{F}E - \dot{E}F}{E^2} \leq \frac{(\Lc_1 S, S)^2 - \|\Lc_1 S\|^2\|S\|^2}{2E^2}
+  C_0(\|H\|_{\infty} + \|S\|_{\infty})(N + C_1),
\end{align*}
and \eqref{eq:nest} follows from H\"older's inequality.
\end{proof}

Now we use Proposition \ref{prop:nest} to obtain a general lower bound for $\|S^{\sigma}(g(t))\|$ and record a simple upper bound in passing. These bounds
in particular imply that a solution to \eqref{eq:mrf} which satisfies \eqref{eq:grs} on some time-slice must be static. Of course, the monotonicity
of $\mu_{\sigma}(g(t))$ already implies
that the future of such a solution is static; that its past is also static can also be deduced from either of \cite{KotschwarRFBU, KotschwarFrequency}.

\begin{proposition}\label{prop:moddecay}
Suppose $g(t)$ is a smooth solution to \eqref{eq:mrf} on $M\times [0, T)$
where $T\in (0, \infty]$. Assume that
\[
\|\Rm(g(t))\|_{\infty} \leq K_0, \quad \|S^{\sigma}(g(t))\|_{\infty} \leq K_1,
\]
and
\begin{equation}\label{eq:shl1}
 \int_0^{T}\left(\|S^{\sigma}(g(t))\|_{\infty} + \|H(t)\|_{\infty}\right)\,dt = \Lambda_0  <\infty.
\end{equation}
Then there are positive constants $N_0 = N_0(K_0, K_1, \Lambda_0, N(0))$ and $N_1 = N_0(K_0, \Lambda_0)$ where
\[
 N(0)  = \left\{\left.\left(\|\nabla S^{\sigma}\|^2 - 2\Rm\left(S^{\sigma}, S^{\sigma}\right)\right)\middle/ \|S^{\sigma}\|^2_{}\right\}\right|_{t=0},
\]
such that
\begin{equation}\label{eq:sdecay}
         e^{-N_0(t+1)}\|S^{\sigma}(g(0))\|^2_{g(0)} \leq  \|S^{\sigma}(g(t))\|_{g(t)}^2 \leq   e^{N_1(t+1)}\|S^{\sigma}(g(0))\|_{g(0)}^2
\end{equation}
for all $t\in [0, T)$. In particular, 
\begin{equation}\label{eq:mudecay}
 \mu_{\sigma}(g(s)) - \mu_{\sigma}(g(t)) \geq 2N_0^{-1}\|S^{\sigma}(g(0))\|^2_{g(0)}e^{-N_0(t+1)}\left(1 - e^{N_0(t-s)}\right) 
\end{equation}
for $0 \leq t \leq s <T$.
\end{proposition}

\begin{proof} Let $S = S^{\sigma}(g(t))$.  For the upper bound in \eqref{eq:sdecay}, note that we have
\begin{align*}
    \dot{E}(t) &=-2F(t) - (HS, S) \leq (4K_0 + \|H(t)\|_{\infty})E(t)
\end{align*}
from \eqref{eq:eev1}, so $E(t) \leq E(0)\exp(4K_0t + \Lambda_0)$. 

For the lower bound, assume first that $E(t) > 0$ on $[0, T)$. Then $N(t)$ is well-defined and by Proposition \ref{prop:nest},
\[
 \dot{N}(t) \leq C_1(\|S(t)\|_{\infty} + \|H(t)\|_{\infty})(N(t) + C_2)
\]
for some $C_i = C_1(K_0, K_1)$, $i=1, 2$. Thus it follows from \eqref{eq:shl1} that 
\[
N(t) \leq N(t) + C_2 \leq e^{C_1\Lambda_0}(N(0) + C_2) \dfn \tilde{N}_0
\]
for all $t\in [0, T)$.
But, from \eqref{eq:eev1},
\[
    (\log E(t))^{\prime} \geq -2N(t) - \|H(t)\|_{\infty} \geq -2\tilde{N}_0 - \|H(t)\|_{\infty},
\]
so $E(t) \geq \exp(-2\tilde{N}_0t - \Lambda_0)E(0)$, and the lower bound in \eqref{eq:sdecay} is valid provided $E(t)\neq 0$.

Suppose then that $E(t_0) = 0$ for some $t_0 \in [0, T)$. We claim that $E(t) \equiv 0$. To see this, note that from the upper bound already proven, we must have $E(t) \equiv 0$ on $[t_0, T)$. Thus we may assume that $t_0$ is the infimum of all $b\geq 0$ for which $E(b) = 0$. If $b >0$, then applying the lower bound proven above to the interval $[0, b)$,
we see that $E(t) \geq e^{-N_0(t+1)}E(0)$ for all $t < b$, and obtain a contradiction sending $t$ to $b$. So we must have $b= 0$ and $E(t) \equiv 0$. The inequality
in \eqref{eq:sdecay} is trivial in this case.

Finally, \eqref{eq:mudecay} follows upon integrating $\mu_{\sigma}^{\prime}(g(t))= 2E(t)$ and applying the left inequality in \eqref{eq:sdecay}.
\end{proof}

\subsection{A remark on the Einstein case}\label{ssec:einsteincase}
It is possible to use a variation on Proposition \ref{prop:moddecay} to give a simpler proof of Theorem \ref{thm:decay} in the special case that 
$(M, \bar{g})$ is Einstein and the convergence of $g(t)$ to $\bar{g}$ is smooth, rather than sequential. In this case, one can work just with the normalized flow \eqref{eq:nrf} and use analogs of the computations above with $f \equiv 0$ to estimate $S^{\sigma}(g(t)) = \Rc(g(t)) + (\sigma/2)g(t)$. Then, if $\int_0^{\infty}\|S^{\sigma}(g(t))\|_{\infty}\,dt  <\infty$ (for example, if $\|g(t) - \bar{g}\|_{C^2_{\bar{g}}} < e^{-bt}$ for some $b$,) it follows that
\[
        \left\|S^{\sigma}(g(t))\right\|_{g(t)}^2 \geq C_0e^{-N_0t}\left\|S^{\sigma}(g(0))\right\|^2_{g(0)}
\]
for some $C_0$, $N_0 > 0$ where here $\|\cdot\|_{g(t)} = \|\cdot\|_{L^2_{g(t)}}$ are now unweighted $L^2$-norms. (Note that, here, $H =  R + \sigma n/2 = \operatorname{tr}_gS^{\sigma}$, so $|H(t)| \leq C|S^{\sigma}(g(t))|$.)

\section{Bounding the rate of convergence from above}

In this section, we convert Theorem \ref{thm:decay} into a problem for the modified Ricci flow. We will first show that for $t\gg 0$, we can pull-back our solution
to \eqref{eq:nrf} by a family of diffeomorphisms to obtain an immortal solution to \eqref{eq:mrf} which converges smoothly to a limit soliton as $t\to \infty$. 
We will assemble our proof out of the \L{}ojasiewicz arguments of
\cite{Ache, HaslhoferMueller, Kroencke, SunWang}, following closely to the template provided by those references. We then prove a refined version of Proposition \ref{prop:moddecay}
to bound the rate of convergence of the solution to the modified flow from above.

First, however, we record a lemma that will allow us to control $H(t)$ (as defined in equation \eqref{eq:hdef}) in terms of $S^{\sigma}(g(t))$ along \eqref{eq:mrf} near a soliton. We use a simple estimate on the lowest nontrivial eigenvalue of the weighted Laplacian
associated to a soliton that has been previously observed, e.g., in \cite{CaoZhu}, \cite{SunWang}. (When $\sigma > 0$
this control on $H(t)$ can be obtained directly from \eqref{eq:heq} and the maximum principle.) 
\begin{lemma}\label{lem:eigest} Let $(M, \bar{g}, \bar{f}, \sigma)$ be a compact soliton and $\alpha \in (0, 1)$. 
There is a regular $C^{2, \alpha}_{\bar{g}}$-neighborhood $U$ of $\bar{g}$ in $\mathcal{R}(M)$ and a constant $C = C(\alpha, \sigma, \bar{g})$
on which
\begin{equation}\label{eq:eigest}
            \|u\|_{C^{2, \alpha}_g} \leq C\|(\Delta_{f_g} - \sigma/2)u\|_{C^{0, \alpha}_g}
\end{equation}
for any $u$ satisfying $\int_M u e^{-f_g}\,dV_g = 0$.
\end{lemma}

\begin{proof} According to Theorem \ref{thm:ls}, there is a $C^{2, \alpha}_{\bar{g}}$-neighborhood $V$ of $\bar{g}$ in $\mathcal{R}(M)$ on which $f_g$ is well-defined
and depends continuously on $g$. On $V$, the lowest nonzero eigenvalue 
\[
    \lambda_g = \inf \bigg\{\, \int_{M}|\nabla u|^2\,e^{-f_g}\,dV_g \;\bigg|\; \int_M u^2\,e^{-f_g}\,dV_g = 1, \,  \int_M u\,e^{-f_g}\,dV_g = 0\, \bigg\}
\]
of $-\Delta_{f_g}$ depends continuously on $g$.   

At the same time, for any $g$ and $f$, there is the following weighted generalization of the Bochner formula
\begin{equation}\label{eq:fbochner}
\Delta_f|\nabla u|^2  = 2|\nabla\nabla u|^2 + 2S^{\sigma}(g, f)(\nabla u, \nabla u) -\sigma|\nabla u|^2 + 2\langle\nabla \Delta_{f} u, \nabla u\rangle. 
\end{equation}
For the soliton $\bar{g}$, $S^{\sigma}(\bar{g}, \bar{f}) = 0$, so, applied to an eigenfunction $\bar{u}$ of $-\Delta_{\bar{f}}$ and integrated over $(M, \bar{g})$ 
relative to $e^{-\bar{f}}dV_{\bar{g}}$, this formula implies (as in \cite{CaoZhu}, \cite{SunWang}) that $\lambda_{\bar{g}} > -\sigma/2$. 
Thus, there is some $C^{2, \alpha}_{\bar{g}}$ neighborhood $U\subset V$ on which 
\[
    \lambda_g + \frac{\sigma}{2} \geq \frac{1}{2}\left(\lambda_{\bar{g}} + \frac{\sigma}{2}\right) > 0
\]
for all $g\in U$, and the claim follows.
\end{proof}

\subsection{Passing to a solution to the modified flow}

The following statement is based on Lemma 3.2 of \cite{SunWang},  Theorem 1 of \cite{HaslhoferMueller}, and Theorem 7.3 of \cite{Kroencke} (see also Sections 2 and 3 of \cite{Ache}), where
the stability hypotheses on the soliton in \cite{HaslhoferMueller} and \cite{Kroencke} are replaced by an priori entropy bound on the solutions to \eqref{eq:nrf}.

\begin{proposition}[\cite{HaslhoferMueller},\cite{Kroencke}, \cite{SunWang}]\label{prop:modflow}
Let $(M, \bar{g})$ be a compact soliton. For any $k\geq 2$, $\alpha\in (0, 1)$, and $\epsilon\in (0, 1)$, there exist $\delta > 0$ and $T_0 > 0$ such that
if $g(t)$ is an immortal solution to \eqref{eq:nrf} with $\|g(0) - \bar{g}\|_{C^{k+2, \alpha}_{\bar{g}}} < \delta$ and  $\mu_{\sigma}(g(t)) \leq \mu_{\sigma}(\bar{g})$
for all $t \geq 0$, then  there is a smooth family $\phi_t\in \operatorname{Diff}(M)$ with $\phi_{T_0} = \operatorname{Id}$
such that $\tilde{g}(t) = \phi_t^*g(t)$ solves \eqref{eq:mrf} on $[T_0, \infty)$ and
\begin{enumerate}
 \item For all $t \geq T_0$, $\|\tilde{g}(t) - \bar{g}\|_{C^{k, \alpha}_{\bar{g}}} < \epsilon$.  
\item \label{it:sder} For each $l \geq 0$,
 \begin{equation*}
\int_{T_0}^{\infty}\|\tilde{\nabla}^{(l)}S^{\sigma}(\tilde{g}(t))\|_{\infty, \tilde{g}(t)}\,dt \leq C_l
\end{equation*}
for some $C_l$ depending on $\bar{g}$ and $\mu_{\sigma}(\bar{g}) - \mu_{\sigma}(g(0))$.
\item As $t\to \infty$, $\gt(t)$ converges smoothly to a soliton $g_{\infty}$ with $\|g_{\infty}-\bar{g}\|_{C^{k, \alpha}_{\bar{g}}} < \epsilon$
and $\mu_{\sigma}(g_{\infty}) = \mu_{\sigma}(\bar{g})$.
\end{enumerate}
\end{proposition}

Up to a few small modifications, the proof follows those for similar assertions in \cite{Ache, HaslhoferMueller, Kroencke, SunWang}. 
We include the details here for completeness.  We will write
\[
 B^{l, \alpha}_{\epsilon} = \left\{\, g\in \mathcal{R}(M)\,\middle|\, \|g - \bar{g}\|_{C^{l, \alpha}_{\bar{g}}} < \epsilon\right\}.
\]

\begin{proof} First, using Theorem \ref{thm:ls} and reducing $\epsilon$ if needed, we may assume that $B^{k, \alpha}_{\epsilon}$ is a regular neighborhood of $\bar{g}$
on which the  \L{}ojasiewicz inequality
\begin{equation}\label{eq:lsineq}
    |\mu_{\sigma}(\bar{g}) - \mu_{\sigma}(g)|^{\theta} \leq C_{L}\|\nabla\mu_{\sigma}(g)\|
\end{equation}
is valid for some $\theta \in [1/2, 1)$. We will specify $\delta$ over the course of the proof.

Let $\bar{\phi}_t\in \operatorname{Diff}(M)$ be the one-parameter family generated by $\delb \bar{f}$. Then $\bar{\phi}_t^*\bar{g}$ solves \eqref{eq:nrf} on $(-\infty, \infty)$. 
Let $T_0 = T_0(\bar{g})$ be equal either to the first positive time that $\bar{\phi}_t^*\bar{g}\notin B_{\epsilon/8}^{k, \alpha}$, or to one, whichever is smaller. By the continuous dependence of the Ricci flow (see Theorem A of \cite{BahuaudGuentherIsenberg}) there is then a $\delta_0 = \delta_0(\epsilon, \bar{g}) > 0$ such that if $g(0) \in B_{\delta_0}^{k+2, \alpha}$, then 
$g(t) \in B_{\epsilon/4}^{k, \alpha}$ for $t\in [0, T_0]$. We will assume that $0< \delta \leq \delta_0$.

The metrics $g(t)$ will at least be regular for $t$ near $T_0$. Taking $\phi_t$ to be the solution to
\[
        \pd{\phi}{t} = -\nabla f_{g(t)} \circ \phi, \quad \phi_{T_0} = \operatorname{Id},
\]
the family of metrics $\tilde{g}(t) = \phi_t^*g(t)$ will solve the modified Ricci flow \eqref{eq:mrf} with $f_{\gt(t)} = f_{g(t)}\circ \phi_t$
for $t$ near $T_0$. The solution $\tilde{g}(t)$ can then be extended in a smooth and canonical fashion as long as $\tilde{g}(t)$ remains in a regular neighborhood of $\bar{g}$. Let $T\in (T_0, \infty]$ denote the maximal time of existence of this smooth extension and let $T_1 \in (T_0, T]$ denote the supremum
of those $T^*\geq T_0$ such that $\gt(t) \in B_{\epsilon/2}^{k, \alpha}$ for $t\in [T_0, T^*]$. 

Now, $g(t) \in B^{k, \alpha}_{\epsilon/4}$ for $t\in [0, T_0]$ and $\|\Rm(g(t)\|_{\infty, g(t)} = \|\Rm(\gt(t))\|_{\infty, \gt(t)}$
for $t\in [T_0, T)$. Thus there is some $K_0 = K_0(\bar{g})$ such that $\|\Rm(g(t))\|_{\infty, g(t)} \leq K_0$ for all $t\in [0, T_1)$. 
Standard derivative estimates for \eqref{eq:nrf} then imply that
\begin{equation}\label{eq:rmderest}
    \|\nabla^{(l)}\Rm(g(t))\|_{\infty, g(t)} \leq K_l, \quad t\in [T_0, T_1),
\end{equation}
for some $K_l = K_l(\bar{g})$, $l\geq 0$. But \eqref{eq:rmderest} is diffeomorphism invariant, so
\begin{equation*}
    \|\delt^{(l)}\Rm(\gt(t))\|_{\infty, \gt(t)} \leq K_l, \quad t\in [T_0, T_1)
\end{equation*}
as well. Since $B_{\epsilon}^{k, \alpha}$ is regular, we then have
\begin{equation*}
        \|\delt^{(l)}\ft\|_{\infty, \gt(t)} \leq K^{\prime}_l, \quad t\in [T_0, T_1),
\end{equation*}
for some constants $K^{\prime}_l = K^{\prime}_l(\bar{g})$ by elliptic regularity, and hence
\begin{equation}\label{eq:sderest}
 \|\delt^{(l)}S^{\sigma}(\tilde{g}(t))\|_{\infty, \gt(t)} \leq \tilde{K}_l, \quad t\in [T_0, T_1),
\end{equation}
for some constants $\tilde{K}_l = \tilde{K}_l(\bar{g})$ for all $l \geq 0$.

We claim that $T_1 = T = \infty$. To see this, suppose $T_1 < \infty$ and consider
\[
Q(t) = \mu_{\sigma}(\bar{g}) - \mu_{\sigma}(\tilde{g}(t)).
\]
On $[T_0, T_1)$, $Q(t)$ is monotone increasing and, since $\mu_{\sigma}(\tilde{g}(t)) = \mu_{\sigma}(g(t))$, we at least have $Q(t)\geq 0$.
However, if $Q(t_0) = 0$ for some $t_0$, then, by monotonicity, $Q(t) \equiv 0$
for all $t\geq t_0$ in which case $\gt(t_0)$ is itself a soliton in $B^{k, \alpha}_{\epsilon/2}$ and $\gt(t)$ is static for all $t\geq t_0$.
(In fact, using Proposition \ref{prop:moddecay} above, it is static for all $t\geq T_0$.) The assertions (1) - (3) are trivially true in this case so, going forward, we may
assume $Q(t) > 0$.

Now, following \cite{SunWang}, fix $\beta \in (2-1/\theta, 1)$ where $\theta$ is as in \eqref{eq:lsineq}. Then $0 < (2-\beta)\theta <  1$ and
\[
    -\frac{d}{dt} Q^{1-(2-\beta)\theta} = 2(1 - (2-\beta)\theta)\|S^{\sigma}(\gt(t))\|^2Q^{-(2-\beta)\theta},
\]
where $\|\cdot\|$ denotes the $L^2$-norm relative to $\gt(t)$ and $e^{-f_{\gt(t)}}dV_{\gt(t)}$.
By \eqref{eq:lsineq}, we have 
\[
Q^{-(2-\beta)\theta}\|S^{\sigma}(\gt(t))\|^{2-\beta}= Q^{-(2-\beta)\theta}\|\nabla\mu_{\sigma}(\tilde{g}(t))\|^{2-\beta} 
\geq C_{L},
\]
so
\[
 -\frac{d}{dt} Q^{1-(2-\beta)\theta} \geq C^{-1}_0\|S^{\sigma}(\gt(t))\|^{\beta},
\]
for some $C_0 = C_0(\beta, \theta, C_{L}) > 0$, and hence
\begin{equation}\label{eq:metdist}
 \int_{a}^{s} \|\partial_t\tilde{g}(t)\|^{\beta}\,dt \leq 2^{\beta}C_0\big\{(\mu_{\sigma}(\bar{g}) -\mu_{\sigma}(\tilde{g}(a)))^{\gamma} - (\mu_{\sigma}(\bar{g}) -\mu_{\sigma}(\tilde{g}(s)))^{\gamma}\big\}
\end{equation}
for $\gamma= 1-(2-\beta)\theta$ and any $T_0 \leq a \leq s < T_1$.

On the other hand, using interpolation inequalities for tensors (see, e.g., \cite{Hamilton3D}, Corollary 12.7) and the estimates 
\eqref{eq:sderest} for $\delt^{l}S^{\sigma}(\gt(t))$, for any $p \geq 1$ we have
\begin{align*}
 \|\partial_t \tilde{g}(t)\|_{L^{2, p}_{\gt(t)}} &\leq C(\beta, p)\|\partial_t \tilde{g}(t)\|^{\beta}_{L^{2}_{\gt(t)}}
 \|S^{\sigma}(\gt(t))\|_{L^{2, N}_{\gt(t)}}^{1 -\beta}
 \leq C(\beta, p, \bar{g}) \|\partial_t \tilde{g}(t)\|^{\beta}
\end{align*}
for some $N = N(p)$. Since the Sobolev constants of metrics in $B^{k, \alpha}_{\epsilon}$ are uniformly controlled, we have 
$\|\partial_t \tilde{g}(t)\|_{C^l_{\tilde{g}(t)}} \leq C(\beta, l, \bar{g})\|\partial_t \tilde{g}(t)\|^{\beta}$ for any $l$.  Thus returning
to \eqref{eq:metdist}, we find that, for any $l\geq 0$ and $T_0 \leq a \leq s < T_1$,
\begin{align}
\begin{split}\label{eq:metdist2}
 \|\tilde{g}(s)- \tilde{g}(a)\|_{C^{l}_{\bar{g}}} &\leq C(l)\int_{a}^{s} \|\partial_t \tilde{g}(t)\|_{C^l_{\bar{g}}}\,dt \leq C(l, \bar{g})\int_{a}^{s}\|\partial_t \tilde{g}(t)\|_{C^l_{\tilde{g}(t)}} \,dt\\  
 &\leq C_1(\mu_{\sigma}(\bar{g}) -\mu_{\sigma}(\tilde{g}(a)))^{\gamma} - C_1(\mu_{\sigma}(\bar{g}) -\mu_{\sigma}(\tilde{g}(s)))^{\gamma}\\
 &\leq C_1(\mu_{\sigma}(\bar{g}) - \mu_{\sigma}(g(0)))^{\gamma}
\end{split}
 \end{align}
for some $C_1 = C_1(\beta, l, \bar{g}, C_{L})$.

There is $\delta_1 = \delta_1(\bar{g}) > 0$ such that if $\|g - \bar{g}\|_{C^2_{\bar{g}}} < \delta_1$, then $|\mu_{\sigma}(g)- \mu_{\sigma}(\bar{g})| < \epsilon/(8C_1)$.
Provided that $\delta$ is chosen that $0 < \delta \leq \min\{\delta_0, \delta_1\}$, we have
\begin{equation}\label{eq:metdist3}
 \|\gt(t) - \bar{g}\|_{C^{k, \alpha}_{\bar{g}}} \leq \|\gt(T_0) - \bar{g}\|_{C^{k, \alpha}_{\bar{g}}} + \|\gt(t) - \gt(T_0)\|_{C^{k, \alpha}_{\bar{g}}} < \frac{3\epsilon}{8},
\end{equation}
for $t\in [T_0, T_1)$. Now, $\mu_{\sigma}(\gt(t)) \to \mu_{\sigma}(g(T_1))$ as $t\to T_1$, so \eqref{eq:metdist2} implies that $\gt(t)$ converges to a smooth metric $\gt(T_1)$ as $t\to T_1$, and \eqref{eq:metdist3} implies that $\gt(T_1) \in B_{\epsilon/2}^{k, \alpha}$. Thus $\gt(t)$ can be extended to a smooth solution to \eqref{eq:mrf}  which
belongs to $B_{\epsilon/2}^{k, \alpha}$  for $t\in [T_0, T^*)$ with some $T^* > T_1$, contradicting the definition of $T_1$. So $T_1 = T = \infty$ and  $\gt(t)$ never exits $B_{\epsilon/2}^{k, \alpha}$.

The estimate \eqref{eq:metdist2} then implies that $\tilde{g}(t)$ converges to a smooth metric $g_{\infty}$ in $B_{\epsilon}^{k, \alpha}$ as $t\to\infty$. Since $S^{\sigma}(\gt(t))\to 0$, $\nabla\mu_{\sigma}(g_{\infty}) = 0$ and
$g_{\infty}$ is a soliton which, in view of \eqref{eq:lsineq}, must satisfy $\mu_{\sigma}(g_{\infty}) = \mu_{\sigma}(\bar{g})$. The estimates on $S^{\sigma}(\gt(t))$ in (\ref{it:sder})
follow from \eqref{eq:metdist2} with $s= \infty$.
\end{proof}

The proof of Proposition \ref{prop:modflow} shows that near a soliton $\bar{g}$, the speed $S^{\sigma}(g(t))$ of a solution to the modified flow and all of its covariant derivatives can potentially be controlled by the entropy. We combine this with Lemma \ref{lem:eigest} to bound the rate of convergence of the entropy from above along the modified flow.
\begin{proposition}\label{prop:moddecay2}
Let $(M, \bar{g})$ be a compact soliton and $g(t)$ an immortal solution to \eqref{eq:mrf} which converges smoothly to $\bar{g}$.
Then there are $C_0$, $N_0 > 0$ such that
\begin{equation}\label{eq:mudecay2}
\mu_{\sigma}(\bar{g}) - \mu_{\sigma}(g(t)) \geq C_0\|\nabla \mu_{\sigma}(g(0))\|^2_{g(0)}e^{-N_0t}
\end{equation}
for all $t\geq 0$.
\end{proposition}
\begin{proof}
 Fix $k\geq 2$ and choose $0< \epsilon <1$ so small that $B_{\epsilon}^{k, \alpha}$ is a regular neighborhood of $\bar{g}$ on which the conclusions of Theorem \ref{thm:ls}
 and Lemma \ref{lem:eigest} are valid. Let $a\in [0, \infty)$ be such that $g(t) \in U$ for all $t\geq a$. Then there are constants $K_l = K_l(\bar{g})$ such that
 \[
        \|\nabla^{l}\Rm(g(t))\|_{\infty} \leq K_l, \quad t\in [a+1, \infty), 
 \]
 and, thus, as in the proof of Proposition \ref{prop:modflow}, constants $\tilde{K}_l = \tilde{K}_l(\bar{g})$ such that
 \[
     \|\nabla^{l}S^{\sigma}(g(t))\|_{\infty} \leq \tilde{K}_l, \quad t\in [a+1, \infty).
 \]
Under our assumptions, $\mu_{\sigma}(g(t)) \nearrow \mu_{\sigma}(\bar{g})$ and $\mu_{\sigma}(g(t)) < \mu_{\sigma}(\bar{g})$
unless $g(t)$ is static, in which case $g(t) \equiv \bar{g}$. As in the case of Proposition \ref{prop:modflow}, we then have 
\[
        \int_{a+1}^{\infty}\|S^{\sigma}(g(t))\|_{C^1_{g(t)}}\,dt \leq C_1(\mu_{\sigma}(\bar{g}) - \mu_{\sigma}(g(0)))^{\gamma}.
\]

Using \eqref{eq:heq} and Lemma \ref{lem:eigest}, we obtain a constant $C_2 > 0$ such that $\|H(t)\|_{\infty} \leq C_2\|S^{\sigma}(g(t))\|_{C^1_{g(t)}}$
for $t\in [a, \infty)$. Hence
\[
  \Lambda_0 = \int_0^{\infty}(\|H(t)\|_{\infty} + \|S^{\sigma}(g(t))\|_{\infty})\,dt < \infty,
\]
and \eqref{eq:mudecay2} follows from Proposition \ref{prop:moddecay} upon sending $s\to \infty$ in \eqref{eq:mudecay}.
\end{proof}

Now we combine Propositions \ref{prop:modflow} and \ref{prop:moddecay2} to prove Theorem \ref{thm:decay}.
\begin{proof}[Proof of Theorem \ref{thm:decay}]
Since $\mu_{\sigma}(g(t_i)) = \mu_{\sigma}(\phi_i^*g(t_i)) \to \mu_{\sigma}(\bar{g})$ as $i\to\infty$ and $\mu_{\sigma}(g(t))$ is monotone increasing, we have  
$\mu_{\sigma}(g(t)) \nearrow \mu_{\sigma}(\bar{g})$ as $t\to \infty$. Suppose
\begin{equation*}
    \liminf_{i\to \infty} e^{mt_i}\|\phi_i^*g(t_i) - \bar{g}\|_{C^2_{\bar{g}}} = 0
\end{equation*}
for all $m\geq 0$.
Then,
 \begin{equation}\label{eq:mudecay3}
  \liminf_{i\to\infty} e^{mt_i}(\mu_{\sigma}(\bar{g}) - \mu_{\sigma}(g(t_i))) = 0
 \end{equation}
for all $m\geq 0$.

 Fix $k\geq 2$, $\alpha\in (0, 1)$, and let $V = B_{\epsilon}^{k, \alpha}$ be a regular neighborhood of $\bar{g}$ on which the conclusions of Theorem \ref{thm:ls} are valid.  Then choose $U = B_{\delta}^{k+2, \alpha}\subset V$ as in Proposition \ref{prop:modflow}. By our assumptions, there is $I$ such that $\phi_i^*g(t_{i})\in U$ for all $i \geq I$. 
Then $\hat{g}(t) = \phi_{I}^*g(t)$ solves \eqref{eq:nrf} for $t \in [0, \infty)$ and $\hat{g}(t_{I})\in U$. 

By Proposition \ref{prop:modflow}, there is an $a\geq t_I$ and
a family of diffeomorphisms $\psi_t$ defined for $t\geq a$ with $\psi_{a} = \operatorname{Id}$, such that $\tilde{g}(t) = \psi_t^*\hat{g}(t)$ solves the modified flow \eqref{eq:mrf}, belongs to $V$ for all $t\geq a$, and 
converges smoothly as $t\to \infty$ to a limit soliton $g_{\infty}$ in $V$ with $\mu_{\sigma}(g_{\infty}) = \mu_{\sigma}(\bar{g})$.

Moreover, by Proposition \ref{prop:moddecay2}, there are constants $C_0$ and $N_0 > 0$ such that
\[
    \mu_{\sigma}(g_{\infty}) - \mu_{\sigma}(\tilde{g}(t)) \geq C_0\|\nabla\mu_{\sigma}(\tilde{g}(a))\|^2_{\tilde{g}(a)}e^{-N_0(t-a)}.
\]
But $\tilde{g}(t) = (\phi_{t_I}\circ\psi_t)^*g(t)$, so $\mu_{\sigma}(\tilde{g}(t)) = \mu_{\sigma}(g(t))$
and 
\[
\mu_{\sigma}(\bar{g}) - \mu_{\sigma}(g(t)) \geq C_0\|\nabla\mu_{\sigma}(\tilde{g}(a))\|^2_{\tilde{g}(a)}e^{-N_0(t-a)}
\]
for $t \geq a$. Invoking \eqref{eq:mudecay3} with $m > N_0$, we see that
\[
\|\nabla\mu_{\sigma}(\tilde{g}(a))\|_{\tilde{g}(a)} = 0,
\]
so $\tilde{g}(a) = \phi_{t_I}^*g(a)$ is a soliton satisfying \eqref{eq:grs}.  

Now, the modified flow beginning from a Ricci soliton is static, so $\tilde{g}(t) = \tilde{g}(a)$ for all $t \geq a$. In particular,
$g_{\infty} = \tilde{g}(a) = \phi_{t_I}^*g(a)$. But then $g(a)= \left(\phi_{t_I}^{-1}\right)^*g_{\infty}$ is also a soliton, and so (e.g., by Lemma 4.3 in \cite{KotschwarRFBU})
$g(t)$ is self-similar on $[0, a]$ and hence on all of $[0, \infty)$.
Thus $g(t) = \varphi_t^* g(0)$ for some family $\varphi_t$ of diffeomorphisms.

Writing $\gamma_i = \varphi_{t_i} \circ \phi_i$, our assumptions imply that $\gamma_i^*g(0)$ converges smoothly to $\bar{g}$ as $i\to \infty$. We may then
extract a subsequence $\gamma_{i_k}$ converging to an isometry $\gamma_{\infty}:(M, g(0)) \to (M,\bar{g})$. Then $(\varphi_t^{-1}\circ \gamma_{\infty})^*g(t) = \bar{g}$ for all $t$.
 \end{proof}

\end{document}